\documentclass[12pt]{article}
\usepackage{amsmath,amsfonts,amssymb,amscd,verbatim,comment, amsthm}
\usepackage{fullpage}
\usepackage{color}
\usepackage{array}
\usepackage{enumerate}
\usepackage{xspace}
\usepackage{bm}
\usepackage[normalem]{ulem}
\usepackage{multirow}
\usepackage{tabularx}
\newcolumntype{C}{>{\centering\arraybackslash}X}

\definecolor{darkgreen}{RGB}{55,138,0}

\newcommand{\ba}{\mathbf{a}}
\newcommand{\bb}{\mathbf{b}}
\newcommand{\bc}{\mathbf{c}}
\newcommand{\bd}{\mathbf{d}}

\newcommand{\bx}{\mathbf{x}}
\newcommand{\by}{\mathbf{y}}

\newcommand{\bzero}{\mathbf{0}}

\newcommand{\kk}{\Bbbk}

\makeatletter
\newcommand{\leqnomode}{\tagsleft@true}
\newcommand{\reqnomode}{\tagsleft@false}
\makeatother

\numberwithin{equation}{section}
\theoremstyle{plain}
 
\newtheorem{theorem}[equation]{Theorem}

\newtheorem{proposition}[equation]{Proposition}

\theoremstyle{definition}
\newtheorem{definition}[equation]{Definition}

\newtheorem{remark}[equation]{Remark}

\newcommand\FF{\mathbb F}

\newcommand\ZZ{\mathbb Z}

\begin{document}

\title{Improved bounds on sizes of generalized caps in $AG(n,q)$}

\date{}

\author{Michael Tait\thanks{Villanova University, Department of Mathematics and Statistics, Villanova, PA 19085, \texttt{michael.tait@villanova.edu}. Michael Tait was supported in part by NSF grant DMS-2011553.} \and Robert Won\thanks{University of Washington, Department of Mathematics, Seattle, WA 98195, \texttt{robwon@uw.edu}. Robert Won was supported in part by an AMS--Simons Travel Grant.} } 


\maketitle

\renewcommand{\thefootnote}{\fnsymbol{footnote}} 
\renewcommand{\thefootnote}{\arabic{footnote}} 

\begin{abstract} 
An $m$-general set in $AG(n,q)$ is a set of points such that any subset of size $m$ is in general position. A $3$-general set is often called a capset. In this paper, we study the maximum size of an $m$-general set in $AG(n,q)$, significantly improving previous results. When $m=4$ and $q=2$ we give a precise estimate, solving a problem raised by Bennett.
\end{abstract}

\section{Introduction}

Throughout, let $\FF_q$ denote the field with $q$ elements and let $n$ be a positive integer. Let $\FF_q^n$ be $n$-dimensional affine space over $\FF_q$, also denoted $AG(n,q)$. The $n$-dimensional projective space over $\FF_q$ is denoted $PG(n,q)$.
A $k$-dimensional affine subspace of a vector space is called a $k$-dimensional \emph{flat}.








An {\em affine combination} of $t$ points $\ba_1, \dots, \ba_t \in \FF_q^n$ is a linear combination $\sum_{j=1}^t c_j \ba_j$ where each $c_j \in \FF_q$ and $\sum_{j=1}^t c_j = 1$. A set of points is called \emph{affinely dependent} if one of the points is an affine combination of the others. Equivalently, $\{\ba_1, \dots, \ba_t\}$ is affinely dependent if there is a linear combination $\sum_{j=1}^t c_j \ba_j = \bzero$ where  $\sum_{j=1}^t c_j = 0$ and the $c_j$ are not identically zero.

\begin{definition}
Let $m$ be a positive integer satisfying $3 \leq m \leq n + 2$ and let $A \subseteq\FF_q^n$ have size at least $m$. Then $A$ is called \emph{$m$-general} if
no $m$ points of $A$ lie on a single $(m-2)$-dimensional flat.
Equivalently, $A$ is $m$-general if and only if any subset of $A$ of size $m$ is in general position.
\end{definition}

\begin{remark}We remark that if $A$ is $m$-general, then it is also $(m-1)$-general, so for each $k = 2, \dots, m$, no $k$ points of $A$ lie on a common $(k-2)$-dimensional flat---that is, every subset of $A$ of size $k \leq m$ is affinely independent.
\end{remark}

The term $m$-general set was introduced by Bennett \cite{B}. We adopt Bennett's terminology in this work, but note that $m$-general sets have been studied under different names. A $3$-general set in $\FF_3^n$ (i.e., when $q = m = 3$) is often called a \emph{cap} or \emph{capset}. 
In \cite{HTW}, an $m$-general set in $\FF_3^n$ was called an \emph{$(m-2)$-cap}.
In this paper, we will work in affine space, but capsets and their generalizations have also been studied in $PG(n,q)$ (see \cite{HS, Thas3, Thas4}). Indeed, as Bennett notes, an $m$-general set in $AG(n,q)$ is essentially the same as what was called an $(|A|; m, m-1; n,q)$-set or simply an $(|A|, m-1)$-set in $PG(n,q)$ in \cite{HS}.

For $m \geq 3$ an integer, let $r_{m}(n,q)$ denote the maximum size of an $m$-general set in $AG(n,q)$. Determining $r_3(n,3)$, the maximum size of a capset in $\FF_3^n$, is a notoriously difficult problem: the exact value of $r_3(n,3)$ is known only when $n \leq 6$ \cite{P4, ELS5, P6}. However, the asymptotic behavior of $r_3(n,3)$ has been studied extensively. In \cite{M}, Meshulam proved that $r_3(n,3)$ is $O(3^n/n)$; this was improved by Bateman and Katz to $O(3^n/n^{1+\epsilon})$ for some fixed $\epsilon$ which is independent of $n$ \cite{BK}.
As in \cite{B}, let
\[ \mu_m(q) = \limsup_{n \to \infty} \frac{\log_q(r_m(n,q))}{n}.
\]
The best-known lower bound for $\mu_3(3)$ is $0.724851$ due to Edel \cite{Elower}. The aforementioned upper bounds on $r_3(n,3)$ do not improve upon the trivial upper bound $\mu_3(3) \leq 1$. 

In recent breakthrough work, Ellenberg and Gijswijt \cite{EG} (adapting a use of the polynomial method from Croot, Lev and Pach \cite{CLP}) proved that $r_3(n,3)$ is $O(2.756^n)$ and therefore that $\mu_3(3) \leq 0.923$. Further, they showed that $\mu_3(q) < 1$ for every prime power $q$. In \cite{B}, Bennett extended the methods of Ellenberg and Gijswijt to obtain bounds on $\mu_m(q)$ for every $m$ and $q$. In particular, Bennett proved the following theorem.

\begin{theorem}[Bennett, {\cite[Theorem 1.2]{B}}]\label{bennett theorem}
Let $n$ be a positive integer, $q$ a prime power, and $m$ an integer such that $3\leq m\leq n+2$. Suppose also that $q$ is odd, or $m$ and $q$ are both even. Then 
\[ r_m(n, q) < 2m + m \cdot \min_{t \in (0,1)} \left( t^{-\frac{q-1}{m}} \cdot \frac{1-t^q}{1-t} \right)^n.
\]
\end{theorem}
\noindent As a corollary, Bennett showed that
\[ \mu_m(q) \leq 1 - \log_q \left(\frac{m^2 - \alpha m + \alpha}{m e^{1-\frac{\alpha}{m}}} \right) + O((q\log q)^{-1})
\]
for some $\alpha$ which depends on $m$.

The case when $q = m = 3$ (the capset problem) has garnered particular interest for several reasons. One reason is because capsets are connected to coding theory (see e.g., \cite{codingbook, HS}). A second reason is the following connection to number theory. In $\FF_3^n$, three points lie on a line if and only if they form a three-term arithmetic progression and so a $3$-general set in $\FF_3^n$ can be equivalently described as a set which contains no three-term arithmetic progressions. The analogous problem of studying subsets $A$ in $\ZZ/N\ZZ$ (or $[N]$) avoiding three-term arithmetic progressions was studied by Roth \cite{Roth} who showed that $|A| = o(N)$. The best-known upper bound is due to Bloom \cite{Bloom} who proved that $|A| = O(N(\log \log N)^4/\log N)$.

Similarly, while the definition of an $m$-general set is inherently geometric, we will show that it admits an arithmetic formulation.
In \cite[Theorem 3.2]{HTW}, it was shown that $4$-general sets in $\FF_3^n$ also have an arithmetic interpretation; namely, $4$-general sets are Sidon sets. A subset $A$ of an abelian group $G$ is called a {\em Sidon set} if the only solutions to $a + b  =  c + d$ with $a,b,c,d \in A$ are the trivial solutions when $(a,b)$ is a permutation of $(c,d)$.

In this paper, we show that $m$-general sets in $\mathbb{F}_q^n$ can be characterized arithmetically for any $q$, $m$, and $n$ in Theorem~\ref{thm.general.eqns}. Using this arithmetic characterization and a counting argument, we give a general upper bound on the size of an $m$-general set in Theorem~\ref{counting upper bound q large}. In Section~\ref{comparision section}, we show that our upper bound significantly improves upon previous results. Finally, in Section~\ref{even section}, we answer a question raised by Bennett \cite{B} in the case when $m=4$ and $q=2$. We conclude by giving some open problems and suggesting a possible approach to making progress on the capset problem in Section~\ref{conclusion}.

\section{Arithmetic conditions for $m$-general sets}

The impetus for this work is to understand the geometric notion of an $m$-general set via arithmetic conditions. As mentioned above, $3$-general sets in $\FF_3^n$ are equivalent to sets which avoid three-term arithmetic progressions. We remark that when $q > 3$, the two notions are not equivalent: e.g., $\{1,2,4\} \subseteq \FF_5$ does not contain a three-term arithmetic progression but is not $3$-general.
In \cite[Theorem 3.2]{HTW}, Huang and the authors showed that $4$-general sets in $\FF_3^n$ are Sidon sets; in this work, it will be useful to consider the following generalization of the notion of a Sidon set.

\begin{definition}
{Let $k$ be a positive integer and let $G$ be an abelian group. A subset $A \subseteq G$ is called a \emph{weak $B_k$ set} if the only solutions to the equation
\[\label{eqn.Bk} a_1 + \cdots + a_k = b_1  + \cdots + b_k 
\]
where $a_1,\dots, a_k$ are distinct elements of $A$ and $b_1,\dots, b_k$ are distinct elements of $A$
are the trivial solutions where $(a_1, \dots, a_k)$ is a permutation of $(b_1, \dots, b_k)$.}

The set $A$ is called simply a \emph{$B_k$ set} if the condition on $a_1, \dots, a_k\in A$ and $ b_1, \dots, b_k\in A$ being distinct is removed from the definition.
\end{definition}

A subset of an abelian group is a Sidon set if and only if it is a $B_2$ set. Understanding the size of $B_k$ sets in the integers is a well-studied problem (see e.g., \cite{Sidon, Ruzsa, CBh, Green, SS} and O'Bryant's Dynamic Survey \cite{OSurvey} on Sidon sets). In \cite{HTW}, by analyzing $B_2$ sets in $\FF_3^n$, we were able to exactly determine $\mu_4(3)$.
We remark that when $q > 3$, a $4$-general set in $\FF_q^n$ is not equivalent to a $B_2$ set, e.g., the subset
\[ \{ (1,0), (0,1), (2,0), (0,2) \} \subseteq \FF_5^2
\]
is a $B_2$ set which fails to be $4$-general.
In general, being a $B_k$ set is a strong condition. For example, a two-element subset $\{\ba, \bb\}$ of $\FF_3^n$ is not a $B_3$ set  since $\ba + \ba + \ba = \bb + \bb + \bb$. There is, however, a relationship between $m$-general sets and \emph{weak} $B_k$ sets, as the following result demonstrates.
\begin{proposition}Let $m \geq 3$ be a positive integer and let $A \subseteq \FF_q^n$. If $A$ is $m$-general, then for all positive integers $1 \leq k \leq \lfloor m/2 \rfloor$, $A$ is a weak $B_k$ set.
\end{proposition}
\begin{proof}
Suppose that $A$ is a subset of $\FF_q^n$ which is not a weak $B_k$ set for some $1 \leq k \leq \lfloor m/2 \rfloor$. Being a weak $B_1$ set is a trivial condition, so we may assume $k \geq 2$. Then there exist distinct $\ba_1,  \dots, \ba_k \in A$ and distinct $\bb_1, \dots, \bb_k \in A$ such that 
\[\ba_1 + \dots + \ba_k = \bb_1 + \dots + \bb_k
\]
and $(\ba_1, \dots, \ba_k)$ is not a permutation of $(\bb_1, \dots, \bb_k)$. Hence, there is some index $1 \leq i \leq k$ such that $\ba_i \neq \bb_j$ for all $1 \leq j \leq k$. Then, we have that
\[ \ba_i = \bb_1 + \cdots + \bb_k - \ba_1 - \cdots - \ba_{i-1} - \ba_{i+1} - \cdots - \ba_k
\]
so $\ba_i$ is a nontrivial affine combination of the rest of the points. Therefore the set \[\{ \ba_1 ,\dots ,\ba_k , \bb_1, \dots, \bb_k\}\]
is not affinely independent and so $A$ is not $m$-general.
\end{proof}

{The converse to the above result does not hold. As an example, in the case that $m = 3$ and $q \geq 3$, a weak $B_1$ set is just a set, while being $3$-general is a nontrivial condition. As mentioned above, when $m = 4$ and $q = 3$, a $4$-general set is a $B_2$ set, which is stronger than being a weak $B_2$ set. Below, we give exact arithmetic conditions which are equivalent to being $m$-general.
In order to do this, we first need some definitions. }

\begin{definition}
Given a function $f(x_1,\dots, x_k): (\FF_q^n)^k \to \FF_q^n$, we say that a subset $A$ of $\FF_q^n$ {\em weakly avoids} the equation $f=\bzero$ if, whenever $\ba_1, \ba_2, \dots, \ba_k$ are $k$ distinct elements of $A$, we have $f(\ba_1, \ba_2, \dots, \ba_k) \neq \bzero$.
\end{definition}


Let $t \geq 3$ be a positive integer and let $\mathbf{c} = (c_1, c_2,\dots, c_t) \in \mathbb{F}_q^{t}$ be a vector of coefficients. Define a function $f_{\bc}: (\FF_q^n)^t \to \FF_q^n$ by
\[
f_\mathbf{c}(x_1,x_2,\dots, x_t) = c_1  x_1 + c_2 x_2 + \cdots + c_t  x_t.
\]
Further define
\[\mathcal{C}_0^{t}:= \{(c_1,\dots, c_t) \in \mathbb{F}_q^{t} :  c_1+\cdots +c_t = 0 \text{ with $c_1, \dots, c_t$ not identically $0$}\},\]
the set of nonzero vectors of coefficients of length $t$ whose entries sum to $0$.



\begin{remark}
With this notation, a set in an abelian group which is three-term arithmetic progression-free is one that weakly avoids the equation \[f_{(1,1,-2)}(x_1,x_2,x_3) = x_1 + x_2 - 2 x_3 = 0
\]
while a weak $B_k$ set weakly avoids the equation
\[ f_{(1,1, \dots, 1, -1, -1, \dots ,-1)}(x_1, \dots, x_{2k}) = 0.
\]
In \cite{Ruzsa}, Ruzsa observed this connection and studied subsets of $[N]$ which avoid the equation $\sum_{j=1}^t c_j x_j = 0$ where the $c_j$ are integers which sum to $0$. The difficulty of these problems is highly variable depending on the vector $(c_1,\dots, c_t)$. For example, the best bounds on the largest size of a subset of $[N]$ avoiding $f_{(2,2,-3,-1)}=0$ are frustratingly far apart \cite{Ruzsa}.
\end{remark}

\begin{theorem}
\label{thm.general.eqns}
Let $m$ be a positive integer satisfying $3 \leq m \leq n$ and let $A \subseteq \FF_q^n$ have size at least $m$. Then $A$ is $m$-general if and only if for all $\mathbf{c}\in \mathcal{C}_0^{m}$ we have that $A$ weakly avoids $f_\mathbf{c}(x_1,\dots,x_m) = \bzero$.
\end{theorem}

\begin{proof}
First, assume that for some $\mathbf{c} = (c_1, \dots, c_m) \in \mathcal{C}_0^{m}$, that $A$ does not weakly avoid $f_\mathbf{c}(x_1,\dots, x_m) = \bzero$. Then there are $m$ distinct elements $\ba_1,\dots, \ba_m \in A$ so that $\sum_{j=1}^m c_j \ba_j = \bzero$ with the $c_j$ not identically zero but $\sum_{j=1}^m c_j = 0$. Therefore, $\{\ba_1, \dots, \ba_m\}$ is affinely dependent so $A$ is not an $m$-general set.

Conversely, assume that $A$ is not an $m$-general set. Then for some $3 \leq t\leq m$, $A$ contains $t$ distinct points $\ba_1, \dots, \ba_t$ which are affinely dependent. That is, there exist coefficients $c_1,\dots, c_t \in \FF_q$ so that $\sum_{j=1}^t c_j \ba_j = \bzero$ with $\sum_{j=1}^t c_j = 0$ and the $c_j$ are not identically zero. Letting $\mathbf{c} = (c_1,\dots, c_t, 0, \dots , 0) \in \FF_q^m$, and choosing any other $m-t$ distinct points $\ba_{t+1}, \dots, \ba_m$ of $A$, we have that $\mathbf{c} \in \mathcal{C}_0^{m}$ and $f_{\bc}(\ba_1, \dots, \ba_m) = \bzero$ so $A$ does not weakly avoid $f_\mathbf{c}(x_1,\dots, x_m) = \bzero$.
\end{proof}

\begin{remark}
We remark that if a set $A$ weakly avoids $f_{\bc}(x_1, \dots, x_m) = \bzero$ for all $\bc \in \mathcal{C}_0^m$, then for all $1 \leq t \leq m$, by considering those coefficient vectors whose last $m-t$ entries are identically zero, we have that $A$ also weakly avoids $f_{\bd}(x_1, \dots, x_t) = \bzero$ for all $\bd \in \mathcal{C}_0^t$
\end{remark}

\section{A bound on the size of an $m$-general set}

{We are now prepared to prove our main result, which is a bound on the size of an $m$-general set in $\FF_q^n$.}

\begin{theorem}\label{counting upper bound q large}
Let $m\geq 4$ be an integer, $q$ be a prime power, and $k = \lfloor m/2 \rfloor$. If $A$ is an $m$-general set in $\mathbb{F}_q^n$ then 
\[
|A| \leq \frac{k\cdot q^{n/k}}{(q-1)^{1-2/k}(q-2)^{1/k}} \quad \text{if} \quad q > 2
\]
and
\[
|A| \leq (k!)^{1/k}\cdot 2^{n/k} + k \quad \text{if} \quad  q = 2.
\]
\end{theorem}
\begin{proof}


First assume that $q>2$. Let $\gamma$ be any fixed non-zero element of $\mathbb{F}_q$ and let $(\mathbb{F}_q, <)$ be any total ordering of $\mathbb{F}_q$. Let $\mathcal{C}_\gamma^*$ be the set of all sequences $(\alpha_1,\dots, \alpha_k) \in \mathbb{F}_q^{k}$ which satisfy $\sum_{j=1}^k \alpha_j = \gamma$ and where all $\alpha_j$ are nonzero.

Let $(\bx_1,\dots, \bx_k)$ be a sequence of $k$ distinct points in $A$ such that $\bx_1 < \bx_2 < \dots < \bx_k$ and let $(\by_1,\dots, \by_k)$ be a sequence of $k$ distinct points in $A$ such that $\by_1 < \by_2 < \dots < \by_k$ (we are not assuming that $\{\bx_1,\dots, \bx_k\}$ is disjoint or distinct from $\{\by_1,\dots, \by_k\}$). Now, since $A$ is an $m$-general set, we must have that $\{\bx_1,\dots, \bx_k, \by_1, \dots, \by_k\}$ is an affinely independent set. We claim that for any $(\alpha_1,\dots, \alpha_k)$ and $(\beta_1,\dots, \beta_k)$ in $\mathcal{C}_\gamma^*$, we must have that 
\begin{equation}\label{equations avoided}
 \sum_{j=1}^k \alpha_j \bx_j \neq  \sum_{j=1}^k \beta_j \by_j,
\end{equation}
unless $(\alpha_1,\dots, \alpha_k) = (\beta_1,\dots, \beta_k)$ and $(\bx_1,\dots, \bx_k) = (\by_1,\dots,\by_k)$. To see this, we consider two cases. In the first case, let $(\bx_1,\dots, \bx_k) = (\by_1,\dots, \by_k)$. Assume that $ \sum_{j=1}^k \alpha_j \bx_j = \sum_{j=1}^k \beta_j \bx_j$. Then we have 
\[
\sum_{j=1}^k (\alpha_j - \beta_j)\bx_j = \bzero.
\]
Since the set $\{\bx_1,\dots, \bx_k\}$ is affinely independent and $\sum (\alpha_j - \beta_j) = 0$, this implies that $\alpha_j = \beta_j$ for all $j$.


In the second case, if $(\bx_1,\dots, \bx_k) \not= (\by_1,\dots, \by_k)$, then 
\[
\sum_{j=1}^k \alpha_j  \bx_j - \sum_{j=1}^k \beta_j  \by_j
\]
is a nontrivial linear combination of at most $2k$ elements of $A$ where the coefficients sum to $0$. Since these points are affinely independent, this implies that 
\[
\sum_{j=1}^k \alpha_j  \bx_j - \sum_{j=1}^k \beta_j  \by_j \neq \bzero.
\]

Now for each $\mathbf{c}: = (\alpha_1,\dots, \alpha_k)\in \mathcal{C}_\gamma^*$, define a function $f_\mathbf{c}: (\mathbb{F}_q^n)^k \to \mathbb{F}_q^n$ by $f_\mathbf{c}(\bx_1,\dots, \bx_k) =  \sum_{j=1}^k \alpha_j \bx_j$. By \eqref{equations avoided}, this implies that 
\[
\left| \bigcup_{\mathbf{c}\in \mathcal{C}_\gamma^*} f_\mathbf{c}(A^k)\right| \geq |\mathcal{C}_\gamma^*|\binom{|A|}{k},
\]
where $f(A^k)$ denotes the image of $f$ restricted to inputs from $A$.

We show that $|\mathcal{C}_\gamma^*| \geq (q-1)^{k-2}(q-2)$. When $k=2$, we will show the inequality in a slightly more general setting where we may allow $\gamma$ to be $0$. To prove the inequality, choose any nonzero $a_1$; then there is a unique $a_2\in \mathbb{F}_q$ such that $a_1+a_2 = \gamma$. Exactly one choice of $a_1$ will correspond to $a_2 = 0$, namely, $a_1 = \gamma$. Therefore, when $\gamma \not= 0$ there are $q-2$ sequences of length $2$ in $\mathcal{C}_\gamma^*$ and when $\gamma=0$ there are $q-1$ sequences of length $2$ in $\mathcal{C}_0^*$, showing the inequality in either case. For larger $k$, choose $a_1,\dots, a_{k-2}$ arbitrarily from $\FF_q\setminus\{0\}$ and let $\eta = \gamma - (a_1 + \dots + a_{k-2})$. By the $k=2$ case {(applied to $\eta$)}, there are at least $q-2$ choices of $a_{k-1}$ and $a_k$ which complete to a sequence in $\mathcal{C}_\gamma^*$.

Since, for each $\bc \in \mathcal{C}_\gamma^*$, the outputs of $f_{\bc}$ live in $\mathbb{F}_q^n$, we have 
\[
(q-1)^{k-2}(q-2)\frac{|A|^k}{k^k} < (q-1)^{k-2}(q-2) \binom{|A|}{k} \leq q^n,
\]
which gives the result when $q>2$.

The proof when $q=2$ is similar but simpler. Every unordered set of $k$ elements of $A$ must have a unique sum, otherwise 
\[
(\mathbf{x}_1 + \cdots + \mathbf{x}_k) - (\mathbf{y}_1 + \cdots + \mathbf{y}_k) = (\mathbf{x}_1 + \cdots + \mathbf{x}_k) + (\mathbf{y}_1 + \cdots + \mathbf{y}_k)= \mathbf{0},
\]
contradicting that $A$ is an $m$-general set. Therefore, we have 
\[
\binom{|A|}{k} \leq 2^n.
\]
Using the inequality $\displaystyle\frac{(|A|-k)^k}{k!} \leq \binom{|A|}{k}$ and rearranging gives the result.
\end{proof}



\section{Comparison with Theorem~\ref{bennett theorem}}\label{comparision section}
In this section, we compare our main theorem to Theorem~\ref{bennett theorem} in two regimes: (1) when $q$ is fixed and $m$ and $n$ tend to infinity and (2) when $m$ is fixed.

Let $B_m(n,q)$ be the upper bound for $r_m(n,q)$ in Theorem~\ref{bennett theorem} (\cite[Theorem 1.2]{B}). 
That is
\[
B_m(n,q) = 2m + m\cdot \min_{t\in (0,1)}\left(t^{-\frac{q-1}{m}}\cdot \frac{1-q^t}{1-q}\right)^n.
\]

\subsection{$q$ fixed and $m, n \to \infty$}

In this section we fix a constant $q$ and consider what happens when $m$ is large. To compare $B_m(n,q)$ with Theorem~\ref{counting upper bound q large} in this regime, we first give a lower bound on $B_m(n,q)$. To do this, let 
\[
h_q(x) = x^{-\frac{q-1}{m}}\cdot\frac{1-x^q}{1-x}.
\]
As noted in \cite[Lemma~3.3]{B}, $h_q(x)$ is a convex function and so its minimum on $(0,1)$ occurs anywhere where its derivative vanishes in the interval. We have (c.f. \cite{B})
\[
h'_q(x) = \frac{x^{-\frac{q-1}{m}-1}}{m(1-x)^2}\cdot \left[(q+m-1)x - (q-1) - x^q((q-1)(m-1)(1-x)+m) \right].
\]

For $m$ larger than a constant which depends only on $q$, we have that 
\begin{align*}
    h'_q\left(\frac{1}{m}\right) & < 0\\
    h'_q\left(\frac{q}{m}\right) & >0.
\end{align*}
It follows that, for $m$ larger than a constant depending only on $q$,
\[
\min_{x\in (0,1)} h_q(x) = \min_{x\in (\frac{1}{m},\frac{q}{m})} h_q(x) \geq \min_{x\in (\frac{1}{m},\frac{q}{m})} x^{-\frac{q-1}{m}} = \left(\frac{q}{m}\right)^{-\frac{q-1}{m}}.
\]

Therefore, for $m$ large enough, we have 
\[
B_m(n,q) \geq m\left(\frac{m}{q}\right)^{(q-1)\frac{n}{m}}.
\]
As $m$ goes to infinity, we have 
\[
\frac{\log_q B_m(n,q)}{n} \geq \frac{\log_q m}{n} + \frac{q-1}{m}\left(\log_q m -1\right) = \Omega\left(\frac{\log_q m}{m}\right).
\]


Another way to say this is that the main result of \cite{B} shows that 
\[
\mu_m(q) \leq (q-1 + o_m(1)) \frac{\log_q m}{m},
\]
whereas Theorem~\ref{counting upper bound q large} shows
\begin{equation}\label{our upper bound m large}
\mu_m(q) \leq \frac{2+o_m(1)}{m}.
\end{equation}
We note that the best-known general lower bound on $\mu_m(q)$ is $\frac{1}{m-1}$, so \eqref{our upper bound m large} gives the correct dependence on $m$.

\subsection{$m \geq 4$ fixed}

We now consider the regime where $m$ is fixed. For small $m$ and small $q$, Bennett computes $B_m(n,q)$ to give bounds for $\mu_m(q)$. For the convenience of the reader, we provide these bounds in Table~\ref{table1}. When $m = 3$, the bounds on $\mu_3(q)$ were first given in \cite{EG}. 

Using Theorem~\ref{counting upper bound q large}, when $m \geq 4$ we obtain a nontrivial bound on $\mu_m(q)$ which is independent of $q$, namely
\[ \mu_m(q) \leq \frac{1}{\lfloor m/2 \rfloor}.
\]
Note that $\mu_m(q) \geq \mu_{m+1}(q)$ since an $(m+1)$-general set is automatically $m$-general. Table~\ref{table2} shows that this bound significantly improves upon previous results.

The fact that our bound on $\mu_m(q)$ is independent of $q$ also provides the first nontrivial bound when $m \geq 4$ is fixed and $q$ tends to infinity.
By \cite[Lemma 3.3]{B},
\[
B_m(n,q) = 2m + m\cdot\left(c_m \cdot q + O(1)\right)^n,
\]
where $c_m$ is a constant depending only on $m$. Therefore, we have that 
\[
\frac{\log_q B_m(n,q)}{n} \geq \frac{\log_q m}{n} + \log_q(c_m + O(1)) + 1 = 1-o_q(1).
\]
This tells us that for $m$ constant the main result of \cite{B} does not improve the trivial bound $\limsup_{q\to \infty} \mu_m(q) \leq 1$, whereas our theorem shows that 
\[
\limsup_{q\to \infty} \mu_m(q) \leq \frac{1}{\lfloor m/2 \rfloor}.
\]





\begin{table}
\centering
\begin{tabular}{ccccccccccc}
\hline
\multicolumn{2}{c}{\multirow{2}{*}{}}                              & \multicolumn{9}{c}{$q$}                               \\ \cline{3-11} 
\multicolumn{2}{c}{}                                               & 2    & 3    & 4    & 5    & 7    & 8    & 9    & 11 & $q \to\infty$   \\ \hline
\multicolumn{1}{c|}{}                     & \multicolumn{1}{c|}{3} &      & .923 &      & .930 & .935 &      & .938 & 0.941 & $1-\log_q(1.188)$ \\
\multicolumn{1}{c|}{}                     & \multicolumn{1}{c|}{4} & .813 & .821 & .829 & .836 & .846 & .851 & .854 & 0.861 & $1-\log_q(1.504)$ \\
\multicolumn{1}{c|}{\multirow{2}{*}{$m$}} & \multicolumn{1}{c|}{5} &      & .735 &      & .756 & .771 &      & .782 & 0.791 & $1-\log_q(1.853)$ \\
\multicolumn{1}{c|}{}                     & \multicolumn{1}{c|}{6} & .651 & .665 & .679 & .690 & .708 & .716 & .722 & 0.734 & $1-\log_q(2.212)$ \\
\multicolumn{1}{c|}{}                     & \multicolumn{1}{c|}{7} &      & .609 &      & .636 & .657 &      & .673 & 0.685 & $1-\log_q(2.577)$ \\
\multicolumn{1}{c|}{}                     & \multicolumn{1}{c|}{8} & .544 & .562 & .577 & .591 & .613 & .622 & .631 & 0.644 & $1-\log_q(2.944)$ \\ \hline
\end{tabular}

\caption{Upper bounds on $\mu_m(q)$ from \cite{B} and \cite{EG}.}
\label{table1}
\end{table}

\begin{table}
\centering
\begin{tabular}{ccc}
\hline
                                          &                        & $q$   any prime power   \\ \hline
\multicolumn{1}{c|}{\multirow{6}{*}{$m$}} & \multicolumn{1}{c|}{3} &             \\
\multicolumn{1}{c|}{}                     & \multicolumn{1}{c|}{4} & $.500$      \\
\multicolumn{1}{c|}{}                     & \multicolumn{1}{c|}{5} & $.500$      \\
\multicolumn{1}{c|}{}                     & \multicolumn{1}{c|}{6} & $.334$      \\
\multicolumn{1}{c|}{}                     & \multicolumn{1}{c|}{7} & $.334$      \\
\multicolumn{1}{c|}{}                     & \multicolumn{1}{c|}{8} & $.250$      \\ \hline
\end{tabular}

\caption{Upper bounds on $\mu_m(q)$ from Theorem~\ref{counting upper bound q large}. }
\label{table2}
    \end{table}

\section{Determining $\mu_4(2)$}\label{even section}

In \cite{B}, Bennett raises the case $q=2$ and $m=4$ as a ``particularly interesting case". As $2$-flats in $\mathbb{F}_2^n$ have exactly $4$ points, finding a $4$-general set in $\mathbb{F}_2^n$ of maximum size is the same as finding the largest subset which does not fully contain a $2$-flat. Bennett shows that $r_4(n,2) < 8 + 4(1.755)^n$, giving $\mu_4(2) < 0.813$. Theorem~\ref{even characteristic} determines $\mu_4(2)$ exactly and in this section we give its proof.

\begin{theorem}\label{even characteristic}
We have 
\[
\frac{1}{\sqrt{2}} 2^{n/2}\leq r_4(n,2) \leq 1+\sqrt{2}\cdot 2^{n/2},
\]
and hence 
\[
\mu_4(2) = \frac{1}{2}.
\]
\end{theorem}

Before we give the proof, we establish some notation. Let $\kk$ be a finite field. A function $f: \kk\to \kk$ is called {\em almost perfect nonlinear} if for any $a,b\in \kk$ with $a\neq 0$, the equation
\[
f(x+a) - f(x) = b
\]
has at most two solutions. When $\kk$ has characteristic $2$, if $x$ is a solution then $x+a$ is also a solution, and so the upper bound of $2$ solutions is best possible. Almost perfect nonlinear functions have applications in cryptography \cite{BScrypt, Nyberg} and coding theory \cite{CCZ}. They can also be used to construct association schemes and strongly regular graphs \cite{vDFDF, vDFDF2} and semi-biplanes \cite{CH}. For more background and applications of almost perfect nonlinear functions, we recommend \cite{maxwellthesis}.

\begin{proof}[Proof of Theorem~\ref{even characteristic}]
The upper bound follows from Theorem \ref{counting upper bound q large} when $m=4$ and $q=2$. 

By Theorem~\ref{thm.general.eqns}, a set $A\subseteq \mathbb{F}_2^n$ is $4$-general if and only if it weakly avoids all equations 
\[
c_1  \bx_1 + c_2   \bx_2 + c_3   \bx_3 + c_4  \bx_4=0,
\]
that satisfy $c_1 + c_2 + c_3 + c_4 = 0$ with the $c_i$ not identically zero. Over $\mathbb{F}_2$, if $c_1 + c_2+c_3+c_4 = 0$, then either exactly two or exactly four of the coefficients are $1$. When exactly two coefficients are $1$, every set weakly avoids the equation, and so determining $r_4(n,2)$ is equivalent to determining the maximum size of a subset $A$ in $\mathbb{F}_2^n$ such that for any $\ba, \bb, \bc, \bd \in A$, if $\ba + \bb = \bc + \bd$ it implies that $\{\ba, \bb\} = \{\bc, \bd\}$ or $\ba = \bb$ and $\bc = \bd$. We will call this a Sidon set but we note that since we are in even characteristic we are also considering $\ba + \ba = \bc + \bc$ to be a trivial solution.


For the lower bound, we construct $4$-general sets using almost perfect nonlinear functions. First let $n$ be even. Since $\mathbb{F}_2^n$ is additively isomorphic to $\mathbb{F}_{2^{n/2}} \times \mathbb{F}_{2^{n/2}}$, we will construct a Sidon set in $\mathbb{F}_{2^{n/2}} \times \mathbb{F}_{2^{n/2}}$. If $f: \mathbb{F}_{2^{n/2}} \to \mathbb{F}_{2^{n/2}}$ is an almost perfect nonlinear function, then the set $A = \{(x, f(x)): x\in \mathbb{F}_{2^{n/2}}\}$ is a Sidon set. To see this (see also \cite{CP}), if 
\[
(x, f(x)) + (y,f(y)) + (w, f(w)) + (z,f(z))= (0,0),
\]
then 
\begin{align*}
x+z &= y+w\\
f(x) + f(z) &= f(y) + f(w).
\end{align*}
If $x=z$ then $y=w$ and this is a trivial solution, so assume that $a\not=0$ is defined so that $x = z+a$. Then $y+w=a$ and (noting that we are in characteristic $2$)
\[
f(z+a) - f(z) = f(w+a) - f(w).
\]
Since $f$ is almost perfect nonlinear, we must have $z = w$ or $z=w+a$, and so assume that $z=w+a$. But this mean $x=z+a = w$ and so this is a trivial solution again. 
    
Thus, if $f$ is almost perfect nonlinear, then there exists a Sidon set of size $2^{n/2}$ in $\mathbb{F}_{2^{n/2}} \times \mathbb{F}_{2^{n/2}}$. Almost perfect nonlinear functions exist for every $n$ \cite{maxwellthesis}. For completeness, we show that $f(x) = x^3$ is almost perfect nonlinear. Fix $a, b \in \FF_{2^{n/2}}$ 
with $a \neq 0$. Since we are in characteristic $2$, the equation
\[
(x+a)^3 - x^3 = b
\]
simplifies to 
\[
ax^2 + a^2x = b.
\]
Since this is quadratic in $x$ it has at most $2$ solutions.
Hence, for even $n$, $r_4(n,2) \geq 2^{n/2}$. Since $r_4(n,2)$ is monotone in $n$, we have proved the lower bound.
\end{proof}

\section{Concluding remarks}\label{conclusion}
In this paper, we improved previous bounds on $r_m(n,q)$ for most choices of $m,n,q$ and gave a precise estimate in the case that $m=4$ and $q=2$. One case that our argument does not cover is the capset problem, when $m=3$. 

One possible approach to making progress on the capset problem is to generalize the function $r_{m}(n,q)$. In extremal graph theory, the {\em Tur\'an number} of a graph $F$ is the maximum number of edges in an $n$ vertex $F$-free graph, and is denoted by $\mathrm{ex}(n, F)$. This was recently systematically generalized by Alon and Shikhelmen \cite{AS} to the function $\mathrm{ex}(n, H, F)$ which denotes the maximum number of copies of a graph $H$ in an $n$ vertex $F$-free graph. This function has been studied extensively since being introduced, helping us to understand the structure of $F$-free graphs. One could analogously study a generalization of $r_{m}(n, q)$ by fixing $s$ and $t$ and asking for the maximum number of sets of size $s$ contained in a $t$-flat which are in a set $A$ that is $m$-general. This combined with a saturation result (that there are many sets of $m$ points not in general position if $A$ larger than $r_{m}(n,q)$) could potentially improve bounds on the capset problem. A similar argument was sketched in \cite{FP2} to give bounds on the maximum size of a subset of $AG(n,3)$ that does not contain an entire $m$-flat.

One way to generalize sets with no $3$-term arithmetic progression in $\mathbb{F}_q^n$ is to ask for the maximum size of a set with no $k$-term arithmetic progression for some $k\geq 3$. This was studied in \cite{LW} and suggests another interesting way to generalize the study of $r_m(n,q)$. Instead of asking for the maximum size of a set avoiding $m$ points in an $(m-2)$-flat, one could fix $m$ and $t\leq m-2$ and ask for the maximum size of a set with no $m$ points that all lie on a $t$-flat.


Above, we showed that for both $m$ fixed and $q \to \infty$ or for $q$ fixed and $m \to \infty$, we have that $\mu_m(q)$ is bounded by roughly $\frac{1}{m-1}$ and $\frac{2}{m}$. Closing this factor of $2$ for any choice of $m$ and $q$ would be very interesting. It is likely that for fixed $m$ and $q$, the lower bound can always be improved. When $m$ is even, it is possible that the upper bound is correct.

In the specific case when $q=2$ and $m=4$ there is also a factor of $2$ difference between the upper and lower bound in Theorem~\ref{even characteristic}. It is reasonable to guess that $r_{4}(n,2) = 2^{n/2}$ when $n$ is even. It would also be interesting to determine better bounds on $r_m(n,2)$ for $m>4$. 








\subsection*{Acknowledgments}
 We thank Craig Timmons for leading us to references on almost perfect nonlinear functions and  Lily Tait for inspiration.

\bibliographystyle{plain}
\bibliography{biblio}

\end{document}